\documentclass[10pt,reqno]{amsart}
\usepackage{a4wide}
\usepackage[bbgreekl]{mathbbol}
\usepackage{mathbbol}
\usepackage[mathscr]{eucal}
\usepackage[all]{xy}
\newtheorem{prop}{Proposition}[section]
\newtheorem{thm}[prop]{Theorem}
\newtheorem{cor}[prop]{Corollary}

\theoremstyle{remark}
\newtheorem{rem}[prop]{Remark}

\newcommand{\cst}{\mathrm{C^*}}
\newcommand{\tens}{\otimes}
\newcommand{\id}{\mathrm{id}}
\newcommand{\CC}{\mathbb{C}}
\newcommand{\RR}{\mathbb{R}}
\newcommand{\TT}{\mathbb{T}}
\newcommand{\NN}{\mathbb{N}}
\newcommand{\sS}{\mathbb{S}}
\newcommand{\ZZ}{\mathbb{Z}}
\newcommand{\eps}{\varepsilon}
\newcommand{\I}{\mathbb{1}}
\newcommand{\gA}{\boldsymbol{A}}
\newcommand{\bC}{\boldsymbol{C}}
\newcommand{\bu}{\boldsymbol{u}}
\newcommand{\bn}{\boldsymbol{n}}
\newcommand{\bepsilon}{\boldsymbol{\epsilon}}
\newcommand{\gothA}{\mathfrak{A}}
\newcommand{\gB}{\mathfrak{B}}
\newcommand{\gX}{\mathfrak{X}}
\newcommand{\gC}{\mathfrak{C}}
\newcommand{\cB}{\mathcal{B}}

\newcommand{\cG}{\mathcal{G}}
\newcommand{\cH}{\mathcal{H}}
\newcommand{\cK}{\mathcal{K}}
\newcommand{\cM}{\mathcal{M}}
\newcommand{\bDelta}{\boldsymbol{\Delta}}
\newcommand{\bPhi}{\boldsymbol{\Phi}}
\newcommand{\cA}{\mathcal{A}}
\newcommand{\comp}{\!\circ\!}
\renewcommand{\Bar}[1]{\overline{#1}}
\newcommand{\st}{\:\vline\:}
\newcommand{\is}[2]{\left(#1\,\vline\,#2\right)}
\newcommand{\tp}{\xymatrix{*+<.7ex>[o][F-]{\scriptstyle\top}}}

\newcommand{\malpha}{\mathbb{\bbalpha}}
\newcommand{\mbeta}{\mathbb{\bbbeta}}
\newcommand{\mgamma}{\mathbb{\bbgamma}}
\newcommand{\mdelta}{\mathbb{\bbdelta}}
\newcommand{\AM}{\mathbb{A}}
\newcommand{\DM}{\mathbb{\Delta}}
\newcommand{\PM}{\mathbb{\Phi}}
\newcommand{\eM}{\mathbb{\bbespilon}}

\DeclareMathOperator{\C}{C}
\DeclareMathOperator{\B}{B}
\DeclareMathOperator{\Mor}{Mor}
\DeclareMathOperator{\Ran}{ran}
\DeclareMathOperator{\spec}{Sp}
\DeclareMathOperator{\M}{M}
\DeclareMathOperator{\QMap}{\mathscr{Q}-Map}
\DeclareMathOperator{\qs}{\mathscr{QS}}

\DeclareMathOperator{\SO}{SO}
\newcommand{\SqU}{\mathrm{S}_q\mathrm{U}(2)}
\newcommand{\SqO}{\mathrm{S}_q\mathrm{O}(3)}
\newcommand{\om}{{\omega_q}}
\newcommand{\ba}{\boldsymbol{a}}
\newcommand{\dA}{{\dot{A}}}
\newcommand{\dC}{{\dot{C}}}
\newcommand{\dG}{{\dot{G}}}
\newcommand{\dK}{{\dot{K}}}
\newcommand{\dL}{{\dot{L}}}
\DeclareMathOperator{\tr}{tr}

\begin{document}

\title{Quantum $\mathrm{SO}(3)$ groups and quantum group actions on $M_2$}

\date{\today}

\author{Piotr M.~So{\l}tan}
\address{Department of Mathematical Methods in Physics\\
Faculty of Physics\\
Warsaw University}
\email{piotr.soltan@fuw.edu.pl}

\thanks{Research partially supported by Polish government grants
no.~115/E-343/SPB/6.PRUE/DIE50/2005-2008 and N201 1770 33.}

\maketitle

\begin{abstract}
Answering a question of Shuzhou Wang we give a description of quantum $\SO(3)$ groups of Podle\'s as universal objects. We use this result to give a complete classification of all continuous compact quantum group actions on $M_2$.
\end{abstract}

\section{Introduction}

The classical group $\SO(3)$ has many useful descriptions. For example $\SO(3)$ is the automorphism group of the $\cst$-algebra $M_2$ of $2\times{2}$ complex matrices. In this paper we will focus on \emph{quantum} $\SO(3)$ groups first defined by Piotr Podle\'s in \cite{spheres}.

In his paper \cite{wang} Shuzhou Wang asked if the quantum $\mathrm{SU}(2)$ group could be described as a quantum automorphism group of $M_2$ endowed with a collection of functionals (\cite[Remark (4) on page 209]{wang}). We solve this problem after a necessary modification. Namely the quantum $\mathrm{SU}(2)$ must be replaced by the quantum $\SO(3)$ group (the confusion stems from an erroneous identification in \cite{wang} of the classical $\mathrm{SU}(2)$ as the automorphism group of $M_2$).

We will show that the quantum groups $\SqO$ of Podle\'s are the universal comact quantum groups acting continuously on $M_2$ and preserving the Powers state
\begin{equation}\label{powers}
\om:M_2\ni\begin{bmatrix}m_{1,1}&m_{2,2}\\m_{2,1}&m_{2,2}\end{bmatrix}
\longmapsto\tfrac{1}{1+q^2}(m_{1,1}+q^2m_{2,2})
\end{equation}
where $q$ is any fixed number in $]0,1]$. In other words $\SqO$ is the quantum automorphism group of the quantum space underlying $M_2$ preserving the state $\om$. More details will be given in subsequent sections.

The quantum $\SO(3)$ groups have so far been defined either as quotient groups of quantum $\mathrm{SU}(2)$ or by means of a complicated system of generators and relations (see Section \ref{PSO3}). Therefore our description of these groups as universal objects in the category of compact quantum groups acting continuously on $M_2$ and preserving the Powers state yields a possible alternative approach to these quantum groups which avoids some technical complexity.

Furthermore the new description of Podle\'s quantum $\SO(3)$ groups gives, together with some other results, a complete description of continuous compact quantum group actions on two by two complex matrices (see Subsection \ref{statement} and Section \ref{last}).

\subsection{Terminology}

\subsubsection{Quantum spaces} We will use the language of quantum groups and quantum spaces (cf.~\cite{pseu,su2,spheres,podles,qs}). A quantum space is, by definition, an object of the category dual to the category of $\cst$-algebras, as defined in \cite[Section 0]{unbo} (see also \cite{pseu}). In this category a morphism from a $\cst$-algebra $\gothA$ to a $\cst$-algebra $\gC$ is a nondegenerate $*$-homomorphism from $\gothA$ to $\M(\gC)$, where $\M(\gC)$ is the multiplier algebra of $\gC$. However, in this paper all $\cst$-algebras will be unital (in other words all quantum spaces will be compact) which implies that morphisms will simply be unital $*$-homomorphisms. We will write $\Mor(\gothA,\gC)$ for the set of all morphisms from $\gothA$ to $\gC$. For a given $\cst$-algebra $\gC$ we will write $\qs(\gC)$ for the corresponding quantum space. Let us stress that by introducin the notion of a quantum space we are not defining a new concept, but simply creating a language in which it seems easier to express some mathematical ideas. Following Piotr Podle\'s (\cite{spheres,podles}) we will sometimes use notation $\C(\gX)$, where $\gX$ is some quantum space. Examples of this are the well known $\cst$-algebras $\C\bigl(\SqO\bigr)$ and $\C\bigl(\SqU\bigr)$ of \cite{spheres,podles} and \cite{su2}.

\subsubsection{Quantum families of maps, quantum semigroups and their actions}
We will consider compact quantum groups and semigroups and their actions on quantum spaces. These are particular cases of \emph{quantum families of maps} defined and studied in detail in \cite{qs}. To explain this notion let us consider three quantum spaces $\qs(N)$, $\qs(M)$ and $\qs(C)$ (in other words let $N,M$ and $C$ be $\cst$-algebras). A \emph{quantum family of maps} $\qs(M)\to\qs(N)$ labeled by $\qs(C)$ is a morphism $\Psi\in\Mor(N,M\tens{C})$. This notions was introduced already in \cite{pseu} and it generalizes the classical notion of a continuous family of continuous maps of locally compact spaces. In case $M$ is finite dimensional and $N$ is finitely generated and unital there exists a special quantum space $\qs(\gC)$ and a quantum family $\Phi\in\Mor(N,M\tens\gC)$ called the \emph{quantum family of all maps} $\qs(M)\to\qs(N)$. It is distinguished by the property that for any $\cst$-algebra $D$ and any $\Psi_{D}\in\Mor(N,M\tens{D})$ there exists a unique $\Lambda\in\Mor(\gC,D)$ such that the diagram 
\begin{equation*}
\xymatrix{
N\ar[rr]^-{\Phi}\ar@{=}[d]&&M\tens\gC\ar[d]^{\id_{M}\tens\Lambda}\\
N\ar[rr]^-{\Psi_{D}}&&M\tens{D}}
\end{equation*}
is commutative. The quantum space $\qs(\gC)$ labeling $\Phi$ is called the \emph{quantum space of all maps} $\qs(M)\to\qs(N)$. In \cite{qs} we show that for $N=M$ the quantum space of all maps $\qs(M)\to\qs(M)$ is a compact quantum semigroup with unit. In other words there exists a comultiplication $\Delta_{\gC}\in\Mor(\gC,\gC\tens\gC)$ and a counit $\epsilon_{\gC}\in\Mor(\gC,\CC)$. Moreover $\Phi$ is then an \emph{action} of $(\gC,\Delta_{\gC})$ on $M$:
\[
(\Phi\tens\id_{\gC})\comp\Phi=(\id_{M}\tens\Delta_{\gC})\comp\Phi.
\]
As in \cite{qs}, the quantum semigroup of all maps $\qs(M)\to\qs(M)$ will be denoted by the symbol $\QMap\bigl(\qs(M)\bigr)$.

In \cite{qs} we also studied some subsemigroups of $\QMap\bigl(\qs(M)\bigr)$. Those important for this paper are the semigroups preserving distinguished states on $M$. Let $\omega$ be a state on $M$ and let $\Psi\in\Mor(M,M\tens{D})$ be a quantum family of maps. We say that $\Psi$ preserves $\omega$ if
\[
(\omega\tens\id_{D})\Psi(m)=\omega(m)\I
\]
for all $m\in{M}$. By \cite[Theorem 5.4]{qs} there exist a special quantum space $\qs(\bC)$ and a quantum family $\bPhi\in\Mor(M,M\tens\bC)$ with the property that for any quantum family $\Psi_D\in\Mor(M,M\tens{D})$ preserving $\omega$ there exists a unique $\Lambda\in\Mor(\bC,D)$ such that
\begin{equation*}
\xymatrix{
M\ar[rr]^-{\bPhi}\ar@{=}[d]&&M\tens\bC\ar[d]^{\id_{M}\tens\Lambda}\\
M\ar[rr]^-{\Psi_{D}}&&M\tens{D}}
\end{equation*}
is commutative. The quantum space $\qs(\bC)$ is naturally endowed with the structure of a quantum semigroup with unit which we denote by $\QMap^\omega\bigl(\qs(M)\bigr)$. Moreover $\bPhi$ is an action of $\QMap^\omega\bigl(\qs(M)\bigr)$ on $\qs(M)$.

Let $(\gB,\Delta_{\gB})$ and $(\gC,\Delta_{\gC})$ be two quantum semigroups. A morphism $\Gamma\in\Mor(\gB,\gC)$ is called a \emph{quantum semigroup morphism} if 
\begin{equation}\label{star}
(\Gamma\tens\Gamma)\comp\Delta_{\gB}=\Delta_{\gC}\comp\Gamma.
\end{equation}
In case $(\gB,\Delta_{\gB})$ and $(\gC,\Delta_{\gC})$ are compact quantum groups (\cite[Definition 2.1]{cqg}) a morphism $\Gamma$ satisfying \eqref{star} is called a \emph{quantum group morphism.}

\subsubsection{Actions of compact quantum groups and the Podle\'s condition}
In this paper we will mainly consider actions of compact quantum \emph{groups} on quantum spaces. Let $\cG=(\gB,\Delta_{\gB})$ be a compact quantum group and let $M$ be a $\cst$-algebra. As in the case of quantum semigroup actions we say that $\Psi_{\gB}\in\Mor(M,M\tens\gB)$ is an action of $\cG$ on $\qs(M)$ if
\[
(\Psi_{\gB}\tens\id_{\gB})\comp\Psi_{\gB}=(\id_{M}\tens\Delta_{\gB})\comp\Psi_{\gB}.
\]
We say that $\Psi_{\gB}$ satisfies \emph{Podle\'s condition} if the set
\[
\bigl\{\Psi_{\gB}(m)(\I\tens{b})\st{m\in{M}},\;b\in\gB\bigr\}
\]
is linearly dense in $M\tens\gB$. This condition was formulated by Piotr Podle\'s in \cite[Definicja 2.2]{podPHD}. It appeared in many later publications by various authors (e.g.~\cite[Definition 1]{boca}, \cite[Definition 3.1]{marciniak}, \cite[Definition 2.6]{vaes}). An even stronger definition of an action was used in \cite[Definition 2.1]{wang}.
Following \cite{vaes} we will call actions of compact quantum groups satisfying the Podle\'s condition \emph{continuous actions.} This condition is always satisfied by actions of classical groups.

An important class of compact quantum group actions are the \emph{ergodic} actions. Since actions are particular cases of quantum families of maps, let us define the notion of an ergodic quantum family. Let $M$ be a $\cst$-algabra and let $\Psi_D\in\Mor(M,M\tens{D})$ be a quantum family of maps $\qs(M)\to\qs(M)$ labelled by a quantum space $\qs(D)$. We say that $\Psi_D$ is \emph{ergodic} if for all $m\in{M}$ the condition that $\Psi_D(m)=m\tens\I$ implies that $m\in\CC\I$. 

\subsection{Statement of main results}\label{statement}

For $q\in[0,1]$ let $\om$ be the state on $M_2$ introduced by \eqref{powers}. When $q>0$, the quantum group $\SqO$ acts continuously on $\qs(M_2)$ preserving $\om$. We denote the morphism describing this action by $\Psi_q$. The core result of this paper can be stated in the following way:

\begin{thm}\label{mainR}
Fix $q\in]0,1]$. Let $\cG=(\gB,\Delta_{\gB})$ be a compact quantum group and let $\Psi_{\gB}\in\Mor(M_2,M_2\tens\gB)$ be a continuous action of $\cG$ on $M_2$ preserving the state $\om$. Then there exists a unique $\Gamma\in\Mor\bigl(\C(\SqO),\gB\bigr)$ such that
\[
(\id_{M_2}\tens\Gamma)\comp\Psi_q=\Psi_{\gB}.
\]
Moreover $\Gamma$ is a quantum group morphism.
\end{thm}

In other words the quantum group $\SqO$ is the universal quantum group acting continuously on $M_2$ and preserving the Powers state $\om$.

In Subsection \ref{q0} treat the case of $q=0$ and show that the universal quantum group acting continuously on $\qs(M_2)$ and preserving $\omega_0$ is the classical group $\TT$. With these results we give a complete description of all continuous compat quantum group actions om $\qs(M_2)$ in Section \ref{last}. A special uniqueness result is given for ergodic actions (cf.~Theorem \ref{dzialania}).

\subsection{A tool from operator theory}

\begin{thm}[Fuglede-Putnam-Rosenbloom]\label{FPR}
Let $\gC$ be a $\cst$-algebra and let $n_1,n_2,a\in\gC$. Assume that $n_1$ and $n_2$ are normal and
that $an_1=n_2a$. Then $an_1^*=n_2^*a$.
\end{thm}

The proof of Theorem \ref{FPR} can be found in \cite[Section 12.16]{rud}. We will only use this theorem in the very special case when $n_2=\lambda{n_1}$ for some fixed $\lambda\in\RR$. Therefore whenever $n$ is a normal element of a $\cst$-algebra $\gC$ and 
\begin{equation*}
an=\lambda{na}
\end{equation*}
for some $a\in\gC$ then
\begin{equation*}
an^*=\lambda{n^*a}.
\end{equation*}

\section{Quantum semigroups preserving states on $M_2$}

\subsection{Quantum semigroup of all maps $\qs(M_2)\to\qs(M_2)$}

The $\cst$-algebra $M_2$ is generated by a single element $\bn$ satisfying $\bn^2=0$, $\bn\bn^*+\bn^*\bn=\I$. We take
\begin{equation*}
\bn=\begin{bmatrix}0&1\\0&0\end{bmatrix}.
\end{equation*}
In what follows we will define morphisms from $M_2$ simply by indicating the image of $\bn$.

The following proposition gives a detailed description of the compact quantum semigroup $\QMap\bigl(\qs(M_2)\bigr)$ together with its action on $\qs(M_2)$.

\begin{prop}\label{QMM2}
The quantum space $\QMap\bigl(\qs(M)\bigr)=(\AM,\DM)$, where $\AM$ is the universal $\cst$-algebra generated by four elements $\malpha,\mbeta,\mgamma$ and $\mdelta$ satisfying the relations:
\begin{equation}\label{nstnkw}
\begin{aligned}
\malpha^*\malpha+\mgamma^*\mgamma+\malpha\malpha^*+\mbeta\mbeta^*&=\I,
&\malpha^2+\mbeta\mgamma&=0,\\
\malpha^*\mbeta+\mgamma^*\mdelta+\malpha\mgamma^*+\mbeta\mdelta^*&=0,
&\malpha\mbeta+\mbeta\mdelta&=0,\\
\mbeta^*\mbeta+\mdelta^*\mdelta+\mgamma\mgamma^*+\mdelta\mdelta^*&=\I,
&\mgamma\malpha+\mdelta\mgamma&=0,\\
\mgamma\mbeta+\mdelta^2&=0.
\end{aligned}
\end{equation}
The quantum semigroup structure on $\QMap\bigl(\qs(M_2)\bigr)$ is given by $\DM\in\Mor(\AM,\AM\tens\AM)$ acting on generators in the following way:
\begin{equation}\label{BigDelta}
\begin{split}
\DM(\malpha)&=
\malpha\malpha^*\tens\malpha+\mbeta\mbeta^*\tens\malpha
+\malpha\tens\mbeta
+\malpha^*\tens\mgamma
+\malpha^*\malpha\tens\mdelta+\mgamma^*\mgamma\tens\mdelta,\\
\DM(\mbeta)&=\malpha\mgamma^*\tens\malpha+\mbeta\mdelta^*\tens\malpha
+\mbeta\tens\mbeta
+\mgamma^*\tens\mgamma
+\malpha^*\mbeta\tens\mdelta+\mgamma^*\mdelta\tens\mdelta,\\
\DM(\mgamma)&=\mgamma\malpha^*\tens\malpha+\mdelta\mbeta^*\tens\malpha
+\mgamma\tens\mbeta
+\mbeta^*\tens\mgamma
+\mbeta^*\malpha\tens\mdelta+\mdelta^*\mgamma\tens\mdelta,\\
\DM(\mdelta)&=\mgamma\mgamma^*\tens\malpha+\mdelta\mdelta^*\tens\malpha
+\mdelta\tens\mbeta
+\mdelta^*\tens\mgamma
+\mbeta^*\mbeta\tens\mdelta+\mdelta^*\mdelta\tens\mdelta,
\end{split}
\end{equation}
while the counit {$\eM$} maps $\malpha,\mgamma$ and $\mdelta$ to $0$ and $\mbeta$ to $1$.

The action $\PM\in\Mor(M_2,M_2\tens\AM)$ of $\QMap\bigl(\qs(M_2)\bigr)$ on $\qs(M_2)$ is given by
\begin{equation}\label{Phibaza}
\PM(\bn)=\bn\bn^*\tens\malpha+\bn\tens\mbeta+\bn^*\tens\mgamma+\bn^*\bn\tens\mdelta=
\begin{bmatrix}\malpha&\mbeta\\\mgamma&\mdelta\end{bmatrix}.
\end{equation}
\end{prop}

The proof of proposition \ref{QMM2} is quite straightforward and can be found in \cite[Proposition 4.1]{kom}.

\subsection{Quantum semigroup preserving Powers states on $M_2$}\label{gA}

Choose $q\in]0,1[$ and let $\om$ be the state on $M_2$ given by
\begin{equation}\label{om}
\om\left(\begin{bmatrix}m_{1,1}&m_{1,2}\\m_{2,1}&m_{2,2}\end{bmatrix}\right)
=\tfrac{1}{1+q^2}(m_{1,1}+q^2m_{2,2}).
\end{equation}

In proposition \ref{Des} we give a detailed description of the quantum semigroup $\QMap^\om\bigl(\qs(M_2)\bigr)$. We shall denote the $\cst$-algebra with comultiplication corresponding to this quantum semigroup by $(\gA,\bDelta)$. The action of $\QMap^\om\bigl(\qs(M_2)\bigr)$ on $\qs(M_2)$ will be described by $\bPhi\in\Mor(M_2,M_2\tens\gA)$ and the counit of $(\gA,\bDelta)$ will be denoted by $\bepsilon$.

\begin{prop}\label{Des}
Let $\QMap^\om\bigl(\qs(M_2)\bigr)=(\gA,\bDelta)$. Then $\gA$ is the universal $\cst$-algebra generated by three elements $\beta$, $\gamma$ and $\delta$ satisfying
\begin{equation}\label{albe1}
\begin{aligned}
q^4\delta^*\delta+\gamma^*\gamma+q^4\delta\delta^*+\beta\beta^*&=\I,
&\beta\gamma&=-q^4\delta^2,\\
\beta^*\beta+\delta^*\delta+\gamma\gamma^*+\delta\delta^*&=\I,
&\gamma\beta&=-\delta^2,\\
\gamma^*\delta-q^2\delta^*\beta+\beta\delta^*-q^2\delta\gamma^*&=0,
&\beta\delta&=q^2\delta\beta,\\
\delta\gamma&=q^2\gamma\delta
\end{aligned}
\end{equation}
and
\begin{equation}\label{albe2}
\begin{split}
q^4\delta\delta^*+\beta\beta^*+q^2\gamma\gamma^*+q^2\delta\delta^*&=\I,\\
q^4\delta^*\delta+\gamma^*\gamma+q^2\beta^*\beta+q^2\delta^*\delta&=q^2\I.
\end{split}
\end{equation}
The comultiplication $\bDelta\in\Mor(\gA,\gA\tens\gA)$ is
\begin{equation}\label{Delta}
\begin{split}
\bDelta(\beta)&=q^4\delta\gamma^*\tens\delta-q^2\beta\delta^*\tens\delta+\beta\tens\beta
+\gamma^*\tens\gamma-q^2\delta^*\beta\tens\delta+\gamma^*\delta\tens\delta,\\
\bDelta(\gamma)&=q^4\gamma\delta^*\tens\delta-q^2\delta\beta^*\tens\delta+\gamma\tens\beta
+\beta^*\tens\gamma-q^2\beta^*\delta\tens\delta+\delta^*\gamma\tens\delta,\\
\bDelta(\delta)&=-q^2\gamma^*\gamma\tens\delta-q^2\delta\delta^*\tens\delta+\delta\tens\beta
+\delta^*\tens\gamma+\beta^*\beta\tens\delta+\delta^*\delta\tens\delta.
\end{split}
\end{equation}
The action of $\QMap^\om\bigl(\qs(M_2)\bigr)$ on $\qs(M_2)$ is given by $\bPhi\in\Mor(M_2,M_2\tens\gA)$ defined by
\begin{equation}\label{Phi}
\bPhi(\bn)=\begin{bmatrix}-q^2\delta&\beta\\\gamma&\delta\end{bmatrix}.
\end{equation}
The counit $\bepsilon$ maps $\gamma$ and $\delta$ to $0$ and $\beta$ to $1$.
\end{prop}

\begin{proof}
We know from \cite[Theorem 5.4]{qs} that $\gA$ is a quotient of $\AM$ by the ideal generated by the set
\begin{equation*}
\bigl\{(\om\tens\id_{\AM})\PM(m)-\om(m)\I\st{m\in{M_2}}\bigr\}.
\end{equation*}
Let $\alpha,\beta,\gamma$ and $\delta$ be images of $\malpha,\mbeta,\mgamma$ and $\mdelta$ under the quotient map $\pi:\AM\to\gA$. Clearly $\alpha,\beta,\gamma,\delta$ generate $\gA$. We have 
\begin{equation}\label{pimor}
\bPhi=(\id_{M_2}\tens\pi)\comp\PM,
\end{equation}
so 
\begin{subequations}
\begin{align}
(\om\tens\id_{\gA})\bPhi(\bn\bn^*)&=\tfrac{1}{1+q^2}
(\alpha\alpha^*+\beta\beta^*+q^2\gamma\gamma^*+q^2\delta\delta^*),\label{on1}\\
(\om\tens\id_{\gA})\bPhi(\bn)&=\tfrac{1}{1+q^2}(\alpha+q^2\delta),\nonumber\\
(\om\tens\id_{\gA})\bPhi(\bn^*)&=\tfrac{1}{1+q^2}(\alpha^*+q^2\delta^*),\nonumber\\
(\om\tens\id_{\gA})\bPhi(\bn^*\bn)&=\tfrac{1}{1+q^2}
(\alpha^*\alpha+\gamma^*\gamma+q^2\beta^*\beta+q^2\delta^*\delta).\label{on2}
\end{align}
\end{subequations}
On the other hand
\begin{equation}\label{on3}
\begin{split}
\om(\bn\bn^*)\I&=\tfrac{1}{1+q^2}\I,\\
\om(\bn^*\bn)\I&=\tfrac{q^2}{1+q^2}\I
\end{split}
\end{equation}
and $\om(\bn)\I=\om(\bn^*)\I=0$. It follows from \cite[Theorem 5.4(1)]{qs} that $(\om\tens\id)\bPhi(m)=\om(m)\I$ for all $m\in{M_2}$. In particular we must have $\alpha=-q^2\delta$ and taking into account the relations \eqref{nstnkw} we see that the generators $\beta,\gamma,\delta$ must satisfy \eqref{albe1}. Relations \eqref{albe2} also follow by considering \eqref{on1}, \eqref{on2} and \eqref{on3}.

By \cite[Theorem 5.4(5)]{qs} the quotient map $\pi$ is a quantum semigroup morphism, so formulas \eqref{Delta} follow directly form \eqref{BigDelta}. Similarly we determine the values of $\bepsilon$. Finally \eqref{Phi} is a consequence of \eqref{Phibaza} and \eqref{pimor}.

So far we know that $\gA$ is a $\cst$-algebra generated by $\beta,\gamma$ and $\delta$ with relations \eqref{albe1} and \eqref{albe2}. To see that it is \emph{the} universal $\cst$-algebra for these relations we use the universal property of $\QMap^\om\bigl(\qs(M_2)\bigr)$. One can easily show that the universal $\cst$-algebra $\widetilde{\gA}$ generated $\beta,\gamma$ and $\delta$ with relations \eqref{albe1} and \eqref{albe2} does admit a morphism $\bPhi$ as defined by \eqref{Phi}. Moreover, since $\{\bn\bn^*,\bn,\bn^*,\bn^*\bn\}$ is a basis for $M_2$, we see that the quantum family of maps $\bPhi\in\Mor\bigl(M_2,M_2\tens\widetilde{\gA}\bigr)$ preserves the state $\om$. It follows that $\gA=\widetilde{\gA}$.
\end{proof}

The symbols 
\begin{equation*}
\gA,\quad\bDelta,\quad\bPhi
\end{equation*}
will from now on be reserved exclusively to denote the objects describing the quantum semigroup structure of $\QMap^\om\bigl(\qs(M_2)\bigr)$ and its action on $\qs(M_2)$. In Section \ref{q10} we will use them also for the case of $q=1$ and $q=0$.

The universal property of $\QMap^\om\bigl(\qs(M_2)\bigr)$ guarantees that for any $\cst$-algebra $\gB$ and any quantum family $\Psi_{\gB}\in\Mor(M_2,M_2\tens\gB)$ preserving the state $\om$ on $M_2$ there exists a unique map $\Lambda\in\Mor(\gA,\gB)$ such that $\Psi_{\gB}=(\id_{M_2}\tens\Lambda)\comp\bPhi$. By \cite[Theorem 5.4(6)]{qs} if $\gB$ admits a comultiplication $\Delta_{\gB}$ such that $\Psi_{\gB}$ satisfies $(\Psi_{\gB}\tens\id_{\gB})\comp\Psi_{\gB}=(\id_{M_2}\tens\Delta_{\gB})\comp\Psi_{\gB}$ then $\Lambda$ is a quantum semigroup morphism, (in fact one need not even demand that $\Delta_{\gB}$ be coassociative, cf.~\cite[Proposition 4.7]{qs}).

\section{Podle\'s quantum $\SO(3)$ groups}\label{PSO3}

In \cite[Remark 3]{spheres} and \cite[Section 3]{podles} Piotr Podle\'s introduced quantum groups $\SqO$ as quotient groups of $\SqU$ for $q\in]-1,1[\setminus\{0\}$. In \cite[Proposition 3.1]{podles} we find the following description of the $\cst$-algebra $\C\bigl(\SqO\bigr)$: 

\begin{prop}[Podle\'s]
The $\cst$-algebra $\C\bigl(\SqO\bigr)$ is the universal $\cst$-algebra generated by five elements
$A,\;C,\;G,\;K,\;L$ satisfying
\begin{center}
\begin{subequations}\label{Prel}
\begin{minipage}{6cm}
\begin{align}
L^*L&=(\I-K)(\I-q^{-2}K),\label{P1}\\
LL^*&=(\I-q^2K)(\I-q^4K),\label{P2}\\
G^*G&=GG^*,\label{P3}\\
K^2&=G^*G,\label{P4}\\
A^*A&=K-K^2,\label{P5}\\
AA^*&=q^2K-q^4K^2,\label{P6}\\
C^*C&=K-K^2,\label{P7}\\
CC^*&=q^2K-q^4K^2,\label{P8}\\
LK&=q^4KL,\label{P9}\\
GK&=KG,\label{P10}
\end{align}
\end{minipage}
\begin{minipage}{6cm}
\begin{align}
AK&=q^2KA,\label{P11}\\
CK&=q^2KC,\label{P12}\\
LG&=q^4GL,\label{P13}\\
LA&=q^2AL,\label{P14}\\
AG&=q^2GA,\label{P15}\\
AC&=CA,\label{P16}\\
LG^*&=q^4G^*L,\label{P17}\\
A^2&=q^{-1}LG,\label{P18}\\
A^*L&=q^{-1}(\I-K)C,\label{P19}\\
K^*&=K.\label{P20}
\end{align}
\end{minipage}
\end{subequations}
\end{center}
\end{prop}

\begin{rem}\label{remPod}
\noindent
\begin{enumerate}
\item\label{remPod1} Using \cite[Lemma 3.2]{podles} one can show that the generators $A,C,G,K,L$ of $\C\bigl(\SqO\bigr)$ satisfy
\begin{equation*}
K=A^*A+G^*G\quad\text{and}\quad{C=q^{-1}LA^*+q^2AG^*}.
\end{equation*}
Therefore one can express every relation from the list \eqref{Prel} using only $A,G$ and $L$.
\item Quantum groups $\SqU$ are defined for the deformation parameter $q\in]-1,1[\setminus\{0\}$. Thus the procedure of taking a quotient by a $\ZZ_2$ action yields $\SqO$ also for negative values of $q$. However, in \cite[Proposition 3.3]{podles} Podle\'s shows that the quantum $\SO(3)$ groups defined for negative deformation parameters are isomorphic to those for positive $q$ (namely $\mathrm{S}_{-q}\mathrm{O}(3)\cong\SqO$). Therefore it is enough to consider only positive $q$ and we will do so in what follows.
\end{enumerate}
\end{rem}

The comultiplication $\Delta_q$ on $\C\bigl(\SqO\bigr)$ acts on generators in the following way:
\begin{equation*}
\begin{split}
\Delta_q(A)&=(\I-q^2K)\tens{A}+A\tens{L}-qA^*\tens{G}-K\tens{A},\\
\Delta_q(C)&=-q^2C\tens{K}+L\tens{C}-qG^*\tens{C^*}+C\tens(\I-K),\\
\Delta_q(G)&=C^*\tens{A}+G\tens{L}-q^{-1}L^*\tens{G}+q^{-2}C^*\tens{A},\\
\Delta_q(K)&=K\tens(\I-q^2K)+q^{-1}A\tens{C}+q^{-1}A^*\tens{C^*}+(\I-K)\tens{K},\\
\Delta_q(L)&=-qC\tens{A}+L\tens{L}+q^2G^*\tens{G}-q^{-1}C\tens{A}.
\end{split}
\end{equation*}

\subsection{Reduction of Podle\'s relations}\label{redu}

The relations \eqref{Prel} contain some superfluous ones. Clearly \eqref{P20} follows immediately from \eqref{P4} and \eqref{P5}, as $K=A^*A+G^*G$. Moreover since $G$ is normal (i.e.~by \eqref{P3}), we have $[G,G^*G]=0$ and by \eqref{P15} and Theorem \ref{FPR}
\begin{equation*}
\begin{split}
AG&=q^2GA,\\
AG^*&=q^2G^*A
\end{split}
\end{equation*}
which can be rewritten as
\begin{equation*}
\begin{split}
GA^*&=q^2A^*G,\\
GA&=q^{-2}AG
\end{split}
\end{equation*}
so that $[G,A^*A]=0$ which implies \eqref{P10}. Finally \eqref{P17} follows from \eqref{P13} and \eqref{P3} by Theorem \ref{FPR}.

\begin{prop}\label{CK}
Let $H$ be a Hilbert space and let $A,C,K\in\B(H)$ be such that $0\leq{K}\leq\I$ and
\begin{subequations}\label{AACCK}
\begin{align}
A^*A=C^*C&=K-K^2,\label{AACCK1}\\
AA^*=CC^*&=q^2K-q^4K^2,\\
AK&=q^2KA,\label{AKq}\\
AC&=CA.\label{przem}
\end{align}
\end{subequations}
Then
\begin{equation}\label{wynik}
CK=q^2KC.
\end{equation}
\end{prop}

Before proving Proposition \ref{CK} let us introduce a (simplified version of a) very convenient notation used by S.L.~Woronowicz e.g.~in \cite{qexp}. Let $\cH$ be a Hilbert space and let $T$ be selfadjoint operator on $\cH$. For a subset $E\subset\spec{K}$ we write $\cH(T\in{E})$ for the spectral subspace of $T$ corresponding to $E$. We extend this notation further by agreeing to write $\cH(a\leq{T}\leq{b})$ for $E=[a,b]$ and $\cH(T=x)$ if $E=\{x\}$. Similarly we write $\cH(T<r)$ for the spectral subspace corresponding to $]-\infty,r[$, etc.

\begin{proof}[Proof of Proposition \ref{CK}]
Let us note that due to \eqref{AACCK} all operators involved preserve not only $\ker{K}$, but also $(\ker{K})^\perp$. Since on $\ker{K}$ the formula \eqref{wynik} holds, we can immediately restrict to the subspace $\cH=(\ker{K})^\perp$. Note that this means that $\ker{A^*}=\ker{K}=\{0\}$.

Let $A=u|A|$ be the polar decomposition of $A$. Then $u$ is a partial isometry with initial subspace $(\ker{|A|})^\perp=\bigl(\ker(K-\I)\bigr)^\perp$ and final subspace $\Bar{\Ran{A}}=(\ker{A^*})^\perp=\cH$. In other words $u$ is a coisometry.

By the adjoint version of \eqref{AKq} we have $KA^*=q^2A^*K$, so
\begin{equation}\label{cancel}
|A|Ku^*=q^2|A|u^*K
\end{equation}
since $|A|$ commutes with $K$ by \eqref{AACCK1}. Moreover, as noted above, the range of $u^*$ is $(\ker{|A|})^\perp$ and $K$ preserves this subspace. Therefore we may cancel $|A|$ in \eqref{cancel} to obtain
\[
q^{-2}Ku^*=u^*K.
\]
Multiplying this relation from the right by $u$ we obtain $q^{-2}Ku^*u=u^*Ku$ or
\begin{equation}\label{napsi}
q^{-2}K\psi=u^*Ku\psi
\end{equation}
for any $\psi\in(\ker{|A|})^\perp=\bigl(\ker(K-\I)\bigr)^\perp$. Let us note the first consequence of \eqref{napsi}. Namely let us take $\eps>0$ and $\psi\in\cH(q^2+\eps\leq{K}\leq1-\eps)$. Then the left hand side of \eqref{napsi} has norm greater or equal to $q^{-2}(q^2+\eps)\|\psi\|=(1+q^{-2}\eps)\|\psi\|$, while the norm of the right hand side is smaller or equal to $\|\psi\|$. Therefore $\psi$ must be $0$. In other words the spectral projection of $K$ corresponding to the interval $]q^2,1[$ is zero. In other words
\begin{equation}\label{key}
\spec{K}\,\cap\,]q^2,1[=\emptyset.
\end{equation}
Let us denote by $\cK$ the space $\bigl(\ker(K-\I)\bigr)^\perp=\Ran{u^*}$. From \eqref{napsi} we gather that
\begin{equation}\label{spektra}
q^{-2}\bigl.K\bigr|_{\cK}=u^*K\bigl.u\bigr|_{\cK}.
\end{equation}
Since $u$ is unitary from $\cK$ onto $\cH$ the spectrum of the right hand side of \eqref{spektra} is equal to $\spec{K}$. The spectrum of the left hand side of \eqref{spektra} is $q^{-2}\bigl((\spec{K})\setminus\{1\}\bigr)$. Taking into account \eqref{key} we find that
\[
\spec{K}\subset\bigl\{1,q^2,q^4,\ldots\bigr\}\cup\bigl\{0\bigr\}.
\]
Let us decompose $\cH$ into direct sum of eigenspaces of $K$:
\[
\cH=\bigoplus_{n=0}^\infty\cH(K=q^{2n}).
\]
In this decomposition we have
\[
\begin{split}
K&=\begin{bmatrix}\quad\;1\quad\;\\&\quad\:q^2\:\quad\\&&\quad\:q^4\:\quad\\&&&\quad\:q^6\:\quad\\&&&&\ddots\end{bmatrix},\\
C^*C=A^*A&=\begin{bmatrix}\quad\;0\;\quad\\&\,q^2-q^4\\&&\,q^4-q^8\\&&&\,q^6-q^{12}\\&&&&\ddots\end{bmatrix},\\
CC^*=AA^*&=\begin{bmatrix}q^2-q^4\\&q^4-q^8\\&&q^6-q^{12}\\&&&q^8-q^{16}\\&&&&\ddots\end{bmatrix}.
\end{split}
\]
Let us also note that $u$ maps $\cH(K=1)$ to $\{0\}$ and is an isometric map of $\cH(K=q^{2n})$ onto $\cH(K=q^{2n-2})$ for $n>0$. Also, since $|A|$ preserves the spectral subspaces of $K$ and is $0$ on $\cH(K=1)$ and invertible on the remaining ones, we see that $A$ maps $\cH(K=1)$ to $\{0\}$ and 
\begin{equation}\label{Amaps}
A\bigl(\cH(K=q^{2n})\bigr)=\cH(K=q^{2n-2})
\end{equation}
for $n>0$.

Let us take $\psi\in\cH(K=q^{2n})$ for some $n>0$. We have $C^*C\psi=(q^{2n}-q^{4n})\psi$, so
\[
(q^{2n}-q^{4n})C\psi=CC^*C\psi
\]
so $C\psi$ is an eigenvector of $CC^*$ with eigenvalue $q^{2n}-q^{4n}$. In other words
\begin{equation}\label{Cmaps}
C\psi\in\cH(K=q^{2n-2})\oplus\cH(K=1-q^{2n-2})
\end{equation}
We will show that in fact
\begin{equation}\label{To}
C\psi\in\cH(K=q^{2n-2}).
\end{equation}
Let us first notice that for generic $q$ the subspace $\cH(K=1-q^{2n-2})$ is $\{0\}$ (e.g.~for $q$ non algebraic or $q<\tfrac{\sqrt{2}}{2}$). Moreover we can assume that $1-q^{2n-2}\neq{q^{2n-2}}$ because in this particular case we already have \eqref{To} (and \eqref{Cmaps} must be modified accordingly).

Assume that $n>1$. Taking into account \eqref{Amaps} and \eqref{Cmaps} as well as \eqref{przem} we obtain
\[
\begin{array}{c@{\;}c@{\;}l@{\smallskip}}
AC\psi&\in&\cH(K=q^{2n-4})\oplus\cH\bigl(K=q^{-2}(1-q^{2n-2})\bigr)\\
\parallel\\
CA\psi&\in&\cH(K=q^{2n-4})\oplus\cH(K=1-q^{2n-4})
\end{array}
\]
Since $1-q^{2n-4}\neq{q^{-2}-q^{2n-4}}$ we see that $AC\psi\in\cH(K=q^{2n-4})$. Looking again at \eqref{Amaps} and \eqref{Cmaps} gives \eqref{To}.

Finally let us deal with the case $n=1$. In this case we have $\cH(K=1-q^{2n-2})=\cH(K=0)=\{0\}$, so \eqref{To} follows from \eqref{Cmaps}. Let $C=v|C|$ be the polar decomposition of $C$. Since $|C|$ preserves eigenspaces of $K$ we have
\begin{equation}\label{ito}
v\bigl(\cH(K=q^{2n})\bigr)=\cH(K=q^{2n-2})
\end{equation}
for $n>0$.

Using \eqref{ito} we compute for $\psi\in\cH(K=q^{2n})$ with $n>0$:
\[
\begin{split}
KC\psi&=Kv|C|\psi=(q^{2n}-q^{4n})Kv\psi=(q^{2n}-q^{4n})q^{2n-2}v\psi,\\
CK\psi&=v|C|K\psi=q^{2n}v|C|\psi=q^2n(q^{2n}-q^{4n})v\psi
\end{split}
\]
so that $CK=q^2KC$ on $\cH(K<1)$. Clearly $CK\psi=q^2KC\psi=0$ for $\psi\in\cH(K=1)$. This ends the proof of \eqref{wynik}.
\end{proof}

Proposition \ref{CK} shows that relation \eqref{P12} follows from \eqref{P5}--\eqref{P8}, \eqref{P11} and \eqref{P16}.

\begin{rem}\label{CKrems}
\noindent\begin{enumerate}
\item For most values of $q$ Proposition \ref{CK} does not require the assumption \eqref{przem}. This is the case e.g. for $q$ strictly smaller than $\tfrac{\sqrt{2}}{2}$ or for those for which $q^{2n}+q^{2m}$ is not equal to $1$ for all $m,n\in\NN$. In particular this is the case for non algebraic $q$.

However, if there exist $m_0,n_0\in\NN$ such that $q^{2n_0}=1-q^{2m_0}$ and $n_0\neq{m_0}$ we can give an example of operators $A,C,K$ on a Hilbert space $H$ satisfying \eqref{AACCK1}--\eqref{AKq} and failing \eqref{wynik}. Indeed, let $H=\ell^2(\ZZ_+)$ with standard orthonormal basis $(e_n)_{n\in\ZZ_+}$. Let
\[
Ke_n=q^{2n}e_n,\quad{s}e_n=\begin{cases}0&n=0,\\e_{n-1}&n>0,\end{cases}\quad
\sigma{e_n}=
\begin{cases}
e_n&n\neq{n_0,m_0},\\
e_{m_0}&n=n_0,\\
e_{n_0}&n=m_0.
\end{cases}
\]
Now putting $K=s\sqrt{K=K^2}$ and $C=s\sigma\sqrt{K-K^2}$ we obtain the required example.
\item\label{CKrem2} Let us note that Proposition \ref{CK} is not true for $q=1$. Indeed one can take $H=L^2\bigl([0,1]\bigr)$ and $K$ the multiplication by the identity function on $[0,1]$. Let $A=\sqrt{K-K^2}$. Moreover let $u$ be the unitary map on $H$ induced by the flip
\begin{equation*}
[0,1]\ni{t}\longmapsto{1-t}\in[0,1]
\end{equation*}
and let $C=uA$. Then $AK=KA$, but $CK=(\I-K)C\neq{KC}$. Moreover we have $AC=CA$.
\end{enumerate}
\end{rem}

\subsection{Action on $\qs(M_2)$}\label{Act}

Let us fix $q\in]0,1[$. The quantum group $\SqO$ acts on the quantum space $\qs(M_2)$. The action is described by the morphism $\Psi_q\in\Mor\bigl(M_2,M_2\tens\C(\SqO)\bigr)$,
\begin{equation}\label{Psiq}
\Psi_q(\bn)=
\begin{bmatrix}
-qA&L\\-qG&q^{-1}A
\end{bmatrix}.
\end{equation}
This action comes from the action of $\SqU$ on $\qs(M_2)$ induced by the fundamental representation of $\SqU$ and is therefore continuous (\cite[Lemma 2.1]{izumi}).

One can check using relations \eqref{Prel} that the action of $\SqO$ preserves the state $\om$ introduced by \eqref{om}. Therefore there exists a unique map $\Lambda_q\in\Mor\bigl(\gA,\C(\SqO)\bigr)$ such that
\begin{equation*}
\Psi_q=(\id_{M_2}\tens\Lambda_q)\comp\bPhi.
\end{equation*}
In particular we have
\begin{equation}\label{Lamq}
\Lambda_q(\beta)=L,\quad\Lambda_q(\gamma)=-qG,\quad\Lambda_q(\delta)=q^{-1}A.
\end{equation}

The symbols
\[
\C\bigl(\SqO\bigr),\quad\Delta_q,\quad\Psi_q,\quad\Lambda_q
\]
will be used throughout the paper in the meaning introduced above.

\section{Characterization of quantum $\SO(3)$ groups}\label{char}

Let $\cG=(\gB,\Delta_{\gB})$ be a compact quantum group with a given continuous action $\Psi_{\gB}$ on $\qs(M_2)$, i.e.~$\Psi_{\gB}\in\Mor(M_2,M_2\tens\gB)$ and
\begin{equation}\label{COACT}
(\Psi_{\gB}\tens\id_{\gB})\comp\Psi_{\gB}=(\id_{M_2}\tens\Delta_{\gB})\comp\Psi_{\gB}.
\end{equation}
Assume further that the action of $\cG$ preserves the state $\om$. Then using the notation introduced in Subsection \ref{gA} there is a unique $\Lambda\in\Mor(\gA,\gB)$ satisfying
\begin{equation*}
\Psi_{\gB}=(\id_{M_2}\tens\Lambda)\comp\bPhi.
\end{equation*}
Note that this implies that
\begin{equation}\label{to}
\Psi_{\gB}(\bn)=(\id_{M_2}\tens\Lambda)\begin{bmatrix}-q^2\delta&\beta\\\gamma&\delta\end{bmatrix}.
\end{equation}

We will prove that $\Lambda$ factorizes uniquely through $\Lambda_q\in\Mor\bigl(\gA,\C(\SqO)\bigr)$ introduced in Subsection \ref{Act}. Let us denote by $\cB$ the canonical Hopf $*$-algebra dense in $\cB$ and let $\epsilon_{\cB}$ and $\kappa_{\cB}$ be its counit and antipode (\cite[Section 2]{cqg}).

\begin{prop}\label{gladkie}
Let
\begin{equation*}
b=\Lambda(\beta),\quad{c=\Lambda(\gamma),}\quad{d=\Lambda(\delta).}
\end{equation*}
Then
\begin{center}
\begin{subequations}\label{str15}
\begin{minipage}{7cm}
\begin{align}
q^4d^*d+c^*c+q^4dd^*+bb^*&=\I,\label{alef}\\
b^*b+d^*d+cc^*+dd^*&=\I,\label{bet}\\
q^4d^*d+c^*c+q^2b^*b+q^2d^*d&=q^2\I,\label{gimel}\\
q^4dd^*+bb^*+q^2cc^*+q^2dd^*&=\I,\label{dalet}\\
c^*d-q^2d^*b+bd^*-q^2dc^*&=0,\label{he}
\end{align}
\end{minipage}
\begin{minipage}{5cm}
\begin{align}
bc&=-q^4d^2,\label{waw}\\
cb&=-d^2,\nonumber\\
bd&=q^2db,\nonumber\\
dc&=q^2cd.\label{tet}
\end{align}
\end{minipage}
\end{subequations}
\end{center}
Moreover $b,c,d\in\cB$ and we have
\begin{equation}\label{gladkiee}
\epsilon_{\cB}(c)=\epsilon_{\cB}(d)=0\quad\text{and}\quad\epsilon_{\cB}(b)=1.
\end{equation}
\end{prop}

\begin{proof}
Relations \eqref{str15} are consequences of relations \eqref{albe1} and \eqref{albe2} satisfied by $\beta$, $\gamma$ and $\delta$. By results of \cite[Theorem 3.6]{marciniak}, \cite[Theorem 1.5]{podles} and \cite[Theorem 6.3]{PodMu} we know that there is a dense ``smooth'' subalgebra $\cM$ of $M_2$ such that $\Psi_{\cB}$ restricted to this subalgebra is a right coaction of the Hopf $*$-algebra $\cB$ on $\cM$. Clearly $\cM=M_2$ and thus $b,c,d\in\cB$ and $\epsilon_{\cB}(c)=\epsilon_{\cB}(d)=0$, $\epsilon_{\cB}(b)=1$.
\end{proof}

We will keep the notation $b,c$ and $d$ for images of $\beta,\gamma$ and $\delta$ under $\Lambda$ throughout the paper. In terms of these elements we have
\begin{equation}\label{PsiB}
\Psi_{\gB}(\bn)=\begin{bmatrix}-q^2d&b\\c&d\end{bmatrix}
\end{equation}
(cf.~\eqref{to}).

\begin{prop}\label{aProp}
Let 
\begin{equation*}
\ba=
\begin{bmatrix}
q^4dd^*+bb^*&-qd&-q^2d^*&q^3d^*d+q^{-1}c^*c\\
qbd^*-q^3dc^*&b&qc^*&c^*d-q^2d^*b\\
db^*-q^2cd^*&q^{-1}c&b^*&q^{-1}d^*c-qb^*d\\
qcc^*+qdd^*&d&qd^*&b^*b+d^*d
\end{bmatrix}.
\end{equation*}
Then the matrix $\ba\in{M_4(\gB)}$ is unitary.
\end{prop}

\begin{proof}
Let $a_{i,j}$ be the $(i,j)$-entry of $\ba$ and define
\begin{equation*}
\begin{aligned}
e_1&=\sqrt{1+q^2}\bn\bn^*,\quad&e_2&=\tfrac{\sqrt{1+q^2}}{q}\bn,\\
e_3&=\sqrt{1+q^2}\bn^*,&e_4&=\tfrac{\sqrt{1+q^2}}{q}\bn^*\bn.
\end{aligned}
\end{equation*}
One can check that
\begin{equation*}
\Psi_{\gB}(e_j)=\sum_{i=1}^4e_i\tens{a_{i,j}}
\end{equation*}
and from this and \eqref{COACT} it follows that
\begin{equation}\label{Deaij}
\Delta_{\gB}(a_{i,j})=\sum_{k=1}^4a_{i,k}\tens{a_{k,j}}.
\end{equation}

Moreover, since $\{e_1,\ldots,e_4\}$ is an orthonormal basis of $M_2$ for the scalar product $\is{m_1}{m_2}=\om(m_1^*m_2)$ for $m_1,m_2\in{M_2}$, we have
\begin{equation*}
\ba^*\ba=\I
\end{equation*}
(cf.~\cite[Proof of Theorem 7.3]{qs}). To prove that $\ba$ is unitary it is therefore enough to show that $\ba$ is right invertible. We know that the entries of $\ba$ belong to $\cB$. In particular we can apply $\kappa_{\cB}$ to these elements. If $\boldsymbol{b}$ is a matrix with entries $\kappa_{\cB}(a_{i,j})$ then $\ba\boldsymbol{b}$ is (by \eqref{Deaij}) a matrix with $\epsilon_{\cB}(a_{i,j})\I$ as the $(i,j)$-entry. It follows from \eqref{gladkiee} that $\ba$ is invertible.
\end{proof}

\begin{thm}\label{main}
There exists a unique $\Gamma\in\Mor\bigl(\C(\SqO),\gB\bigr)$ such that $\Lambda=\Gamma\comp\Lambda_q$.
\end{thm}

\begin{proof}
Let
\begin{equation}\label{dotted}
\dA=qd,\quad\dC=bd^*-q^2dc^*,\quad\dG=-q^{-1}c,\quad\dK=q^2d^*d+q^{-2}c^*c,\quad\dL=b.
\end{equation}
Using relations \eqref{alef}, \eqref{gimel}, \eqref{dalet} and \eqref{he} one can show that the matrix $\ba$ defined in Proposition \ref{aProp} can be expressed using $\dA,\dC,\dG,\dK$ and $\dL$:
\begin{equation}\label{nowa}
\ba=\begin{bmatrix}
\I-q^2\dK&-\dA&-q\dA^*&q\dK\\
q\dC&\dL&-q^2\dG^*&-\dC\\
\dC^*&-\dG&\dL^*&-q^{-1}\dC^*\\
q\dK&q^{-1}\dA&\dA^*&\I-\dK
\end{bmatrix}.
\end{equation}
Our aim is to use the unitarity of $\ba$ and the remaining relations \eqref{str15} to show that the elements \eqref{dotted} satisfy the Podle\'s relations \eqref{Prel}. 

By the results of Subsection \ref{redu} we do not need to check relations \eqref{P10}, \eqref{P17} and \eqref{P20}. Moreover \eqref{P13}, \eqref{P14}, \eqref{P15} and \eqref{P18} follow immediately from \eqref{waw}--\eqref{tet}. Relation \eqref{P12} will be shown the moment we verify \eqref{P5}--\eqref{P8}, \eqref{P11} and \eqref{P16}, by Proposition \ref{CK}.

To make calculations easier let us list some of the relations following from the unitarity of $\ba$. Considering the $(1,2)$, $(1,3)$, $(2,2)$, $(3,3)$ $(4,3)$ and $(4,4)$ entries of $\ba\ba^*$ we obtain
\begin{subequations}
\begin{align}
q(\I-q^2\dK)\dC^*-\dA\dL^*+q^3\dA^*\dG-q\dK\dC^*&=0,\label{dwa12}\\
(\I-q^2\dK)\dC+\dA\dG^*-q\dA^*\dL-\dK\dC&=0,\label{dwa13}\\
q^2\dC\dC^*+\dL\dL^*+q^4\dG^*\dG+\dC\dC^*&=\I,\label{dwa22}\\
\dC^*\dC+\dG\dG^*+\dL^*\dL+q^{-2}\dC^*\dC&=\I,\label{dwa33}\\
q\dK\dC-q^{-1}\dA\dG^*+\dA^*\dL-q^{-1}(\I-\dK)\dC&=0,\label{dwa43}\\
q^2\dK^2+q^{-2}\dA\dA^*+\dA^*\dA+(\I-\dK)^2&=\I.\label{dwa44}
\end{align}
\end{subequations}
Similarly considering the $(2,2)$ and $(4,4)$ entries of $\ba^*\ba$ we obtain
\begin{subequations}
\begin{align}
\dA^*\dA+\dL^*\dL+\dG^*\dG+q^{-2}\dA^*\dA&=\I,\label{jeden22}\\
q^2\dK^2+\dC^*\dC+q^{-2}\dC\dC^*+(\I-\dK)^2&=\I.\label{jeden44}
\end{align}
\end{subequations}
Finally let us rewrite \eqref{bet} in terms of elements \eqref{dotted}:
\begin{equation}\label{A2}
\dL^*\dL+q^{-2}\dA^*\dA+q^2\dG\dG^*+q^{-2}\dA\dA^*=\I.
\end{equation}

\noindent\underline{Step 1: normality of $\dG$.} The matrix $\ba\in{M_4(\cB)}$ is unitary and its inverse is $(\id_{M_4}\tens\kappa_{\cB})\ba$. Therefore we know the values of $\kappa_{\cB}$ on matrix elements of \eqref{nowa}:
\begin{equation*}
\begin{aligned}
\kappa_{\cB}(\dA)&=-q\dC^*,&\kappa_{\cB}(\dA^*)&=-q^{-1}\dC,\\
\kappa_{\cB}(\dC)&=-q^{-1}\dA^*,&\kappa_{\cB}(\dC^*)&=-q\dA,\\
\kappa_{\cB}(\dG)&=q^2\dG,&\kappa_{\cB}(\dG^*)&=q^{-2}\dG^*,\\
\kappa_{\cB}(\dL)&=\dL^*,&\kappa_{\cB}(\dL^*)&=\dL,\\
\kappa_{\cB}(\dK)&=\dK.
\end{aligned}
\end{equation*}
By antimultiplicativity of $\kappa_{\cB}$ we have
\begin{equation*}
\kappa_{\cB}(\dG^*\dG)=\dG\dG^*\quad\text{and}\quad\kappa_{\cB}(\dA^*\dA)=\dC^*\dC.
\end{equation*}
Therefore
\begin{equation}\label{GG}
\dA^*\dA+\dG^*\dG=\dK=\kappa_{\cB}(\dK)=\dC^*\dC+\dG\dG^*.
\end{equation}
Therefore, by \eqref{jeden22} and \eqref{dwa33} we have 
\begin{equation}\label{AC}
\dA^*\dA=\dC^*\dC,
\end{equation}
and so, by \eqref{GG} again we have
\begin{equation}\label{GGGG}
\dG^*\dG=\dG\dG^*.
\end{equation}
We have thus checked the relation \eqref{P3}.

\noindent\underline{Step 2: consequences.} Now we note that normality of $\dG$  implies additional commutation relations with $\dL$ and $\dA$. By Theorem \ref{FPR} we have
\begin{subequations}
\begin{align}
\dL\dG^*&=q^4\dG^*\dL,\label{LGgw}\\
\dA\dG^*&=q^2\dG^*\dA.\label{AGgw}
\end{align}
\end{subequations}
Moreover \eqref{jeden44} and \eqref{dwa44} together with \eqref{AC} give
\begin{equation}\label{AC2}
\dC\dC^*=\dA\dA^*.
\end{equation}
From \eqref{jeden44} we have
\begin{equation}\label{str332}
\dA^*\dA+q^{-2}\dA\dA^*=2\dK-\dK^2-q^2\dK^2
\end{equation}
while inserting \eqref{AC} into \eqref{dwa33} and using \eqref{A2} gives
\begin{equation}\label{str331}
\dA^*\dA-q^{-2}\dA\dA^*=(q^2-1)\dG\dG^*=(q^2-1)\dG^*\dG.
\end{equation}
Adding \eqref{str331} and \eqref{str332} gives
\begin{equation}\label{naG}
2\dA^*\dA=q^2\dG^*\dG-q^2\dK^2+2\dK-\dG^*\dG-\dK^2.
\end{equation}
Since $\dK=\dG^*\dG+\dA^*\dA$, from \eqref{naG} we get
\begin{equation}\label{KGG}
\dK^2=\dG^*\dG,
\end{equation}
so \eqref{P4} is verified. As a consequence \eqref{naG} gives
\begin{equation}\label{AAKK}
\dA^*\dA=\dK-\dK^2
\end{equation}
and thus \eqref{P5} and \eqref{P7} are checked. Now inserting \eqref{AAKK} into \eqref{str331} quickly gives
\begin{equation}\label{AAKK2}
\dA\dA^*=q^2\dK-q^4\dK^2,
\end{equation}
so we get \eqref{P6} and (by \eqref{AC2}) \eqref{P8}.

Inserting \eqref{KGG}, \eqref{AAKK} and \eqref{AAKK2} into \eqref{A2} gives quickly
\begin{equation*}
\dL^*\dL=(\I-\dK)(\I-q^{-2}\dK).
\end{equation*}
Similarly inserting \eqref{AC2} into \eqref{dwa22} and using \eqref{AAKK} and \eqref{AAKK2} gives
\begin{equation*}
\dL\dL^*=(\I-q^2\dK)(I-q^4\dK).
\end{equation*}
This means that we checked relations \eqref{P1} and \eqref{P2}.

\noindent\underline{Step 3: commutation of $\dA$ and $\dC$.} By \eqref{dotted} we have
\begin{equation*}
\dC=q^{-1}\dL\dA^*+q^2\dA\dG^*.
\end{equation*}
Therefore, using \eqref{str331}, \eqref{AGgw}, \eqref{GGGG} and analogs for $\dA$, $\dG$ and $\dL$ of relations \eqref{P14} and \eqref{P18} we compute
\begin{equation*}
\begin{split}
\dC\dA&=(q^{-1}\dL\dA^*+q^2\dA\dG^*)\dA\\
&=q^{-1}\dL\dA^*\dA+q^2\dA\dG^*\dA\\
&=q^{-1}\dL\bigl(q^{-2}\dA\dA^*+(q^2-1)\dG^*\dG\bigr)+q^2\dA\dG^*\dA\\
&=q^{-3}\dL\dA\dA^*+(q-q^{-1})\dL\dG^*\dG+\dA^2\dG^*\\
&=q^{-3}\dL\dA\dA^*+(q-q^{-1})\dL\dG\dG^*+\dA^2\dG^*\\
&=q^{-1}\dA\dL\dA^*+(q-q^{-1})\dL\dG\dG^*+\dA^2\dG^*\\
&=q^{-1}\dA\dL\dA^*+(q^2-1)\dA^2\dG^*+\dA^2\dG^*\\
&=q^{-1}\dA\dL\dA^*+q^2\dA^2\dG^*=\dA\dC,
\end{split}
\end{equation*}
which verifies relation \eqref{P16}.

\noindent\underline{Step 4: finishing touches.} Let us now address relation \eqref{P11}. We have by \eqref{str331} and \eqref{AGgw}
\begin{equation*}
\begin{split}
\dK\dA&=(\dA^*\dA+\dG^*\dG)\dA\\
&=\dA^*\dA\dA+\dG^*\dG\dA\\
&=\bigl(q^{-2}\dA\dA^*+(q^2-1)\dG^*\dG\bigr)\dA+\dG^*\dG\dA\\
&=q^{-2}\dA\dA^*\dA+(q^2-1)\dG^*\dG\dA+\dG^*\dG\dA\\
&=q^{-2}\dA\dA^*\dA+q^2\dG^*\dG\dA\\
&=q^{-2}\dA\dA^*\dA+q^{-2}\dA\dG^*\dG=q^{-2}\dA\dK.
\end{split}
\end{equation*}
In addition to \eqref{P11} we now also have \eqref{P12} (cf.~Proposition \ref{CK}).

Let us rewrite \eqref{dwa12} and \eqref{dwa13} using \eqref{AGgw} as
\begin{equation*}
\begin{split}
\dC-q^2\dC\dK-q^{-1}\dL\dA^*+q^2\dG^*\dA-\dC\dK&=0,\\
\dC-q^2\dK\dC-\dA^*\dL+\dA\dG^*-\dK\dC&=0.
\end{split}
\end{equation*}
Consequently
\begin{equation*}
q^2(\dK\dC-\dC\dK)+q\dA^*\dL-q^{-1}\dL\dA^*+(\dK\dC-\dC\dK)=0.
\end{equation*}
Now note that
\begin{equation}\label{stage}
q^{-1}\dL\dA^*+q^2\dA\dG^*=\dC=q\dA^*\dL+\dG^*\dA
\end{equation}
by \eqref{dotted} and \eqref{he}. Therefore, using \eqref{AGgw}, we obtain
\begin{equation*}
q^{-1}\dL\dA^*-q\dA^*\dL=(1-q^4)\dG^*\dA.
\end{equation*}
Thus \eqref{stage} can be rewritten as
\begin{equation*}
(1+q^2)(\dK\dC-\dC\dK)=(1-q^4)\dG^*\dA.
\end{equation*}
Now using \eqref{P12} we obtain $\dK\dC=\dG^*\dA$ and again by \eqref{AGgw} we get
\begin{equation*}
q^{-1}\dA\dG^*=q\dK\dC.
\end{equation*}
Plugging this into \eqref{dwa43} gives
\begin{equation*}
\dA^*\dL=q^{-1}(\I-\dK)\dC,
\end{equation*}
so \eqref{P19} is verified.

The last relation from the list \eqref{Prel} which remains to be checked  is \eqref{P9}. In order to do it we compute using \eqref{stage}, \eqref{LGgw}, \eqref{AGgw}, \eqref{P18}, \eqref{P14} and \eqref{P13}:
\begin{equation*}
\begin{split}
\dL\dK&=\dL(\dA^*\dA+\dG^*\dG)=\dL\dA^*\dA+\dL\dG^*\dG\\
&=(q^2\dA^*\dL+q\dG^*\dA-q^3\dA\dG^*)\dA+\dL\dG^*\dG\\
&=q^2\dA^*\dL\dA+q\dG^*\dA\dA-q^3\dA\dG^*\dA+\dL\dG^*\dG\\
&=q^2\dA^*\dL\dA+q\dG^*\dA\dA-q^3\dA\dG^*\dA+q^4\dG^*\dL\dG\\
&=q^2\dA^*\dL\dA+q\dG^*\dA\dA-q^5\dG^*\dA\dA+q^4\dG^*\dL\dG\\
&=q^2\dA^*\dL\dA+\dG^*\dL\dG-q^4\dG^*\dL\dG+q^4\dG^*\dL\dG\\
&=q^4\dA^*\dA\dL+q^4\dG^*\dG\dL=q^4\dK\dL
\end{split}
\end{equation*}
which means that $\dA,\dC,\dG,\dK$ and $\dL$ satisfy all relations from the list \eqref{Prel}. Therefore there exists a unique $\Gamma:\C\bigl(\SqO\bigr)\to\gB$ such that
\begin{equation*}
\Gamma(A)=\dA,\quad\Gamma(C)=\dC,\quad\Gamma(G)=\dG,\quad\Gamma(K)=\dK,\quad\Gamma(L)=\dL.
\end{equation*}
By \eqref{Lamq} we have $\Gamma\comp\Lambda_q=\Lambda$. Since $\C\bigl(\SqO\bigr)$ is generated by $A,G$ and $L$, this condition determines $\Gamma$ uniquely.
\end{proof}

\begin{cor}
Let $\Gamma\in\Mor\bigl(\C(\SqO),\gB\bigr)$ be the map defined in Theorem \ref{main}. Then 
\begin{enumerate}
\item\label{pierwszy} $(\id_{M_2}\tens\Gamma)\comp\Psi_q=\Psi_{\gB}$ and this property determines $\Gamma$ uniquely.
\item\label{drugi} $\Delta_{\cB}\comp\Gamma=(\Gamma\tens\Gamma)\comp\Delta_q$.
\end{enumerate}
\end{cor}

\begin{proof}
The formula from statement \eqref{pierwszy} follows easily from \eqref{Psiq} and \eqref{PsiB}. Note that applying this formula to the generator $\bn\in{M_2}$ and comparing matrix elements fixes $\Gamma$ on elements $A,G,L$ generating $\C\bigl(\SqO\bigr)$ (cf.~Remark \ref{remPod}\eqref{remPod1}). 

To prove the second statement we compute
\begin{equation}\label{DeDe}
\begin{split}
\bigl(\id_{M_2}\tens\bigl[(\Gamma\tens\Gamma)\comp\Delta_q)\bigr]\bigr)\comp\Psi_q
&=(\id_{M_2}\tens\Gamma\tens\Gamma)\comp(\id_{M_2}\tens\Delta_q)\comp\Psi_q\\
&=(\id_{M_2}\tens\Gamma\tens\Gamma)\comp(\Psi_q\tens\id_{\C(\SqO)})\comp\Psi_q\\
&=\bigl(\bigl[(\id_{M_2}\tens\Gamma)\comp\Psi_q\bigr]\tens\Gamma\bigr)\comp\Psi_q\\
&=(\Psi_{\gB}\tens\id_{\gB})\comp\Psi_{\gB}\\
&=(\id_{M_2}\tens\Delta_{\gB})\comp\Psi_{\gB}\\
&=(\id_{M_2}\tens\Delta_{\gB})\comp(\id_{M_2}\tens\Gamma)\comp\Psi_q\\
&=\bigl(\id_{M_2}\tens[\Delta_{\gB}\comp\Gamma]\bigr)\comp\Psi_q.
\end{split}
\end{equation}
Again applying both sides of \eqref{DeDe} to $n$ and comparing matrix elements yields statement \eqref{drugi}.
\end{proof}

We have thus proved Theorem \ref{mainR} for $q\in]0,1[$.

\begin{rem}
Let $u$ be the fundamental representation of the quantum group $\SqU$ (\cite[Theorem 1.4 and \S 5]{su2}). Then $S=u\tp{u}$ is a four dimensional representation of $\SqU$ which factorizes through $\SqO$. In other words its matrix elements belong to the $\cst$-subalgebra $\C\bigl(\SqO\bigr)$ of $\C\bigl(\SqU\bigr)$. Conjugating this matrix with
\begin{equation*}
V=\begin{bmatrix}
0&-1&0&0\\
1&0&0&0\\
0&0&0&-1\\
0&0&1&0
\end{bmatrix}
\end{equation*}
we obtain
\begin{equation*}
VSV^*=\begin{bmatrix}
\I-q^2K&-A&-qA^*&qK\\
qC&L&-q^2G^*&-C\\
C^*&-G&L^*&-q^{-1}C^*\\
qK&q^{-1}A&A^*&\I-K
\end{bmatrix}
\end{equation*}
(cf. \eqref{nowa}).
\end{rem}

\section{Cases of $q=1$ and $q=0$}\label{q10}

\subsection{Case $q=1$}\label{q1} The state $\om$ on $M_2$ becomes the normalized trace $\tr$ when we put $q=1$. In this subsection we will show that the universal compact quantum group acting on $M_2$ and preserving the trace is the classical group $\SO(3)$. This is contrary to \cite[Remark on page 203]{wang}. Although for $q=1$ the quantum group $\SqO$ is isomorphic to the classical $\SO(3)$ (\cite[Remark 3]{spheres}), we cannot follow directly the path of Section \ref{char} because some of the tools we used there are no longer applicable to the case $q=1$ (e.g.~Proposition \ref{CK}, cf.~Remark \ref{CKrems}\eqref{CKrem2}).

The description of $\QMap^{\tr}\bigl(\qs(M_2)\bigr)$ is the following.

\begin{prop}\label{Des1}
Let $\QMap^{\tr}\bigl(\qs(M_2)\bigr)=(\gA,\bDelta)$. Then $\gA$ is the universal $\cst$-algebra generated by three elements $\beta$, $\gamma$ and $\delta$ satisfying
\begin{equation*}
\begin{aligned}
\delta^*\delta+\gamma^*\gamma+\delta\delta^*+\beta\beta^*&=\I,
&\beta\gamma&=-\delta^2,\\
\beta^*\beta+\delta^*\delta+\gamma\gamma^*+\delta\delta^*&=\I,
&\gamma\beta&=-\delta^2,\\
\delta^*\delta+\gamma^*\gamma+\beta^*\beta+\delta^*\delta&=\I,
&\beta\delta&=\delta\beta,\\
\delta\delta^*+\beta\beta^*+\gamma\gamma^*+\delta\delta^*&=\I,
&\delta\gamma&=\gamma\delta\\
\gamma^*\delta-\delta^*\beta&+\beta\delta^*-\delta\gamma^*=0.
\end{aligned}
\end{equation*}
The comultiplication $\bDelta\in\Mor(\gA,\gA\tens\gA)$ is
\begin{equation*}
\begin{split}
\bDelta(\beta)&=\delta\gamma^*\tens\delta-\beta\delta^*\tens\delta+\beta\tens\beta
+\gamma^*\tens\gamma-\delta^*\beta\tens\delta+\gamma^*\delta\tens\delta,\\
\bDelta(\gamma)&=\gamma\delta^*\tens\delta-\delta\beta^*\tens\delta+\gamma\tens\beta
+\beta^*\tens\gamma-\beta^*\delta\tens\delta+\delta^*\gamma\tens\delta,\\
\bDelta(\delta)&=-\gamma^*\gamma\tens\delta-\delta\delta^*\tens\delta+\delta\tens\beta
+\delta^*\tens\gamma+\beta^*\beta\tens\delta+\delta^*\delta\tens\delta.
\end{split}
\end{equation*}
The action of $\QMap^{\tr}\bigl(\qs(M_2)\bigr)$ on $\qs(M_2)$ is given by $\bPhi\in\Mor(M_2,M_2\tens\gA)$ defined by
\begin{equation*}
\bPhi(\bn)=\begin{bmatrix}-\delta&\beta\\\gamma&\delta\end{bmatrix}.
\end{equation*}
The counit $\bepsilon$ maps $\gamma$ and $\delta$ to $0$ and $\beta$ to $1$.
\end{prop}

Let $\cG=(\gB,\Delta_{\gB})$ be a compact quantum group acting on $M_2$ preserving the state $\tr$ and let $\Psi_{\cB}\in\Mor(M_2,M_2\tens\gB)$ be this action. There exists a unique $\Lambda\in\Mor(\gA,\gB)$ such that $(\id_{M_2}\tens\Lambda)\comp\bPhi=\Psi_{\gB}$. As in Proposition \ref{gladkie}  we let
\[
b=\Lambda(\beta),\quad{c=\Lambda(\gamma),}\quad{d=\Lambda(\delta).}
\]
Then we have
\[
\Psi_{\gB}(\bn)=\begin{bmatrix}-d&b\\c&d\end{bmatrix}
\]
and
\begin{equation}\label{commu}
\begin{aligned}
d^*d+c^*c+dd^*+bb^*&=\I,&bc&=-d^2,\\
b^*b+d^*d+cc^*+dd^*&=\I,&cb&=-d^2,\\
d^*d+c^*c+b^*b+d^*d&=\I,&bd&=db,\\
dd^*+bb^*+cc^*+dd^*&=\I,&dc&=cd\\
c^*d-d^*b&+bd^*-dc^*=0.
\end{aligned}
\end{equation}
Moreover the matrix
\[
\begin{bmatrix}
dd^*+bb^*&-d&-d^*&d^*d+c^*c\\
bd^*-dc^*&b&c^*&c^*d-d^*b\\
db^*-cd^*&c&b^*&d^*c-b^*d\\
cc^*+dd^*&d&d^*&b^*b+d^*d
\end{bmatrix}.
\]
describing the action of $\cG$ on elements of the basis $\{\bn\bn^*,\bn,\bn^*,\bn^*\bn\}$ is unitary. Now the same reasoning as presented in the proof of Theorem \ref{main} shows that $c$ is normal.

Using the normality of $c$ and relations \eqref{commu} we easily show that $b$ and $d$ are also normal. Then, using Theorem \ref{FPR} we show that $b,c,d$ and their adjoints all commute. Moreover $bc=-d^2$ and
\[
2|d|^2+|b|^2+|c|^2=\I.
\]

Let us define the following subset of $\CC^3$:
\[
\sS=\left\{\begin{bmatrix}s\\t\\r\end{bmatrix}\in\CC^3\st
\begin{array}{r@{\;}c@{\;}l}st&=&-r^2,\\|s|&+&|t|=1\end{array}\right\}.
\]
It is an amusing exercise to show that $\sS$ is homeomorphic to $\mathbb{RP}(3)$ and the multiplication
\[
\begin{bmatrix}s\\t\\r\end{bmatrix}\cdot\begin{bmatrix}s\\t\\r\end{bmatrix}=
\begin{bmatrix}2(r\Bar{t}-s\Bar{r})r'+ss'+\Bar{t}t'\\
2(t\Bar{r}-r\Bar{s})r'+ts'+\Bar{s}t'\\
\bigl(|s|^2-|t|^2\bigr)r'+rs'+\Bar{r}t'\end{bmatrix}
\]
gives this space a locally compact group structure with which $\sS$ is isomorphic to $\SO(3)$. The unit element is
\[
\begin{bmatrix}1\\0\\0\end{bmatrix}
\]
and the inverse is described by
\[
\begin{bmatrix}s\\t\\r\end{bmatrix}^{-1}=
\begin{bmatrix}\Bar{s}\\t\\x\end{bmatrix},
\]
where
\[
x=\begin{cases}\tfrac{|s|t}{r}&r\neq{0},\\0&\text{otherwise.}\end{cases}
\]

Using the above description of $\SO(3)$ we can view $\C\bigl(\SO(3)\bigr)$ and its standard comultiplication in the following way: $\C\bigl(\SO(3)\bigr)$ is the universal $\cst$-algebra generated by three normal and commuting elements $S,T$ and $R$ satisfying the relations
\[
ST=-R^2,\quad|S|+|T|=\I.
\]
The comultiplication acts on generators in the following way
\[
\begin{split}
S&\longmapsto2(RT^*-SR^*)\tens{R}+S\tens{S}+T^*\tens{T},\\
T&\longmapsto2(TR^*-RS^*)\tens{R}+T\tens{S}+S^*\tens{T},\\
R&\longmapsto(S^*S-T^*T)\tens{R}+R\tens{S}+R^*\tens{T}.
\end{split}
\]

Now we see that there exists a unique $\Gamma\in\Mor\bigl(\C(\SO(3)),\gB)$ such that
\[
\Gamma(S)=b,\quad\Gamma(T)=c,\quad\Gamma(R)=d.
\]
One can check that this map intertwines the standard action of $\SO(3)$ on $M_2$ with $\Psi_{\gB}$. Moreover this condition determines $\Gamma$ uniquely. We have thus proved the following result

\begin{thm}
Let $\cG=(\gB,\Delta_{\gB})$ be a compact quantum group and let $\Psi_{\gB}\in\Mor(M_2,M_2\tens\gB)$ be a continuous action of $\cG$ on $M_2$ preserving the trace. Then there exists a unique $\Gamma\in\Mor\bigl(\C(\SO(3)),\gB\bigr)$ such that
\[
(\id_{M_2}\tens\Gamma)\comp\Psi_1=\Psi_{\gB},
\]
where $\Psi_1\in\Mor\bigl(M_2,M_2\tens\C(\SO(3))\bigr)$ is the morphism describing the standard action of $\SO(3)$ on $M_2$. Moreover $\Gamma$ is a compact quantum group morphism.
\end{thm}

\subsection{Case $q=0$}\label{q0} Since there is no obvious definition of $\mathrm{S}_0\mathrm{O}(3)$, let us first describe the compact quantum semigroup $\QMap^\om\bigl(\qs(M_2)\bigr)$ for $q=0$. Note that in this case the state $\om$ is not faithful. We can easily find the $\cst$-algebra with comultiplication $(\gA,\bDelta)$ describing this quantum semigroup.

\begin{prop}\label{Des0}
For $q=0$ the $\cst$-algebra $\gA$ is the universal $\cst$-algebra generated by two elements $\beta$ and $\delta$ satisfying the following relations
\begin{equation}\label{re0}
\begin{aligned}
\beta\beta^*&=\I,&\delta^2&=0,\\
\beta\delta&=0,&\beta\delta^*&=0,\\
&\qquad\beta^*\beta+\delta^*\delta+\delta\delta^*=\I.
\end{aligned}
\end{equation}
The colmultiplication $\bDelta$ acts on generators in the following way:
\[
\begin{split}
\bDelta(\beta)&=\beta\tens\beta,\\
\bDelta(\delta)&=\delta\tens\beta+\beta^*\beta\tens\delta+\delta^*\delta\tens\delta,
\end{split}
\]
while the counit maps $\beta$ to $1$ and $\delta$ to $0$.

The action of $\QMap^{\om}\bigl(\qs(M_2)\bigr)$ on $\qs(M_2)$ is given by $\bPhi\in\Mor(M_2,M_2\tens\gA)$ defined by
\[
\bPhi(\bn)=\begin{bmatrix}0&\beta\\0&\delta\end{bmatrix}.
\]
\end{prop}

The proof of Proposition \ref{Des0} is completely analogous to that of Proposition \ref{Des}.

\begin{rem}\label{dn0}
Let us note that the relations \eqref{re0} can be realized inthe following way: let $H$ be an infinite dimensional Hilbert space and let $U$ be a unitary map $H\to{H}\oplus{H}\oplus{H}$. Then we can put
\[
\begin{split}
\beta:H\oplus{H}\oplus{H}\ni\begin{bmatrix}x\\y\\z\end{bmatrix}&\longmapsto{Ux}
\in{H}\oplus{H}\oplus{H},\\
\delta:H\oplus{H}\oplus{H}\ni\begin{bmatrix}x\\y\\z\end{bmatrix}&\longmapsto
\begin{bmatrix}0\\0\\y\end{bmatrix}
\in{H}\oplus{H}\oplus{H}.
\end{split}
\]
Note that in this case $\delta\beta\neq{0}=\beta\delta$, $\delta^*\beta\neq{0}=\beta\delta^*$ and $\beta^*\beta\neq\I=\beta\beta^*$. In particular $\delta\neq0$ in $\gA$ and $\gA$ is not commutative. In fact it clearly contains a copy of the Toepliz algebra.
\end{rem}

\begin{prop}\label{ncqg0}
The quantum semigroup $\QMap^{\omega_0}\bigl(\qs(M_2)\bigr)$ is not a compact quantum group.
\end{prop}

\begin{proof}
Assume that $(\gA,\bDelta)$ is a compact quantum group. The element $Y=\beta^*\beta$ is a group-like projection. If the Haar measure of $(\gA,\bDelta)$ were faithful, $Y$ would have to belong to the dense Hopf $*$-algebra $\cA$ of $\gA$ (by \cite[Theorem 2.6(2)]{cqg}). In that case the coinverse of $Y$ would have to be its inverse. However, we know that in $\gA$ the element $Y$ is a proper projection (cf.~Remark \ref{dn0}), and as such cannot be invertible.

This means that the Haar measure of $(\gA,\bDelta)$ is not faithful and the image of $Y$ under the reducing map $\lambda:\gA\to\gA_r$ (\cite[Page 656]{pseudogr}) is either $\I$ or $0$.

If $\lambda(Y)=\I$ then $\lambda(\delta)$ must be $0$ and $\gA_r$ is generated by a single unitary $\lambda(\beta)$. In that case, however, the $\cst$-algebra $\gA_r$ is commutative and thus the quantum group $(\gA_r,\bDelta_r)$ is not only reduced, but also universal.
Therefore we must have $\delta=0$ in $\gA$ which is not true by Remark \ref{dn0}.

The only remaining possibility is that $\lambda(Y)=0$. In this case $\lambda(\beta)=0$ and $\gA_r$ is generated by $x=\lambda(\delta)$ which satisfies
\[
x^2=0\quad\text{and}\quad{xx^*+x^*x=\I.}
\]
It follows that $\gA_r=M_2$. However this $\cst$-algebra does not admit a compact quantum group structure. This contradiction shows that the assumption that $(\gA,\bDelta)$ was a compact quantum group was false.
\end{proof}

Before continuing let us introduce the action of the classical group $\TT$ on $M_2$ preserving $\omega_0$. Let $\bu$ be the standard generator of $\C(\TT)$ then the morphism $\Psi_0\in\Mor\bigl(M_2,M_2\tens\C(\TT)\bigr)$
\begin{equation}\label{acu}
\Psi_0(\bn)=\begin{bmatrix}0&\bu\\0&0\end{bmatrix}
\end{equation}
describes the action by automorphism sending for each $\mathrm{e}^{\mathrm{i}\varphi}\in\TT$ the element $\bn$ to $\mathrm{e}^{\mathrm{i}\varphi}\bn$. Now a very similar reasoning to that given in the proof of Proposition \ref{ncqg0} leads to the following result:

\begin{thm}
Let $\cG=(\gB,\Delta_{\gB})$ be a compact quantum group and let $\Psi_{\gB}\in\Mor(M_2,M_2\tens\gB)$ be a continuous action of $\cG$ on $\qs(M_2)$ preserving the state $\omega_0$. Then there exists a unique $\Gamma\in\Mor\bigl(\C(\TT),\gB)$ such that
\begin{equation}\label{ac0P}
(\id_{M_2}\tens\Gamma)\comp\Psi_0=\Psi_{\gB}.
\end{equation}
\end{thm}

\begin{proof}
We keep the notation from Proposition \ref{Des0}. Let $\Lambda\in\Mor(\gA,\gB)$ be the unique morphism satisfying
\[
(\id_{M_2}\tens\Lambda)\comp\bPhi=\Psi_{\gB}
\]
and let
\[
b=\Lambda(\beta),\quad{d}=\Lambda(\delta).
\]
In terms of $b$ and $d$ the action of $\cG$ on $\qs(M_2)$ is given by
\begin{equation}\label{ac0}
\Psi_{\gB}(\bn)=\begin{bmatrix}0&b\\0&d\end{bmatrix}.
\end{equation}

We know that $b$ and $d$ belong to the dense Hopf $*$-subalgebra $\cB$ of $\gB$. Moreover
\[
\begin{split}
\Delta_{\gB}(b)&=b\tens{b},\\
\Delta_{\gB}(d)&=d\tens{b}+b^*b\tens{d}+d^*d\tens{d}.
\end{split}
\]
As noted in the proof of Proposition \ref{ncqg0} $b^*b$ is a group-like projection in $\cB$, and so it is either $0$ of or $\I$. 

The possibility $b=0$ can be excluded in many ways. Either we show that in that case $d^*d$ is a proper projection which is also group-like or we use the fact that $d^2=0$ implies that $\epsilon_{\cB}(d)=0$, so that $d^*d=(\id_{\cB}\tens\epsilon_{\cB})(d^*d\tens{d})=0$ which is impossible.

This means that $b$ is unitary and $d=0$. Consequently there exists a unique $\Gamma\in\Mor\bigl(\C(\TT),\gB)$ such that
\begin{equation}\label{bub}
\Gamma(\bu)=b.
\end{equation}
Comparing \eqref{acu} and \eqref{ac0} shows that \eqref{bub} is equivalent to \eqref{ac0P}.
\end{proof}

\section{General actions on $M_2$}\label{last}

Let $\cG=(\gB,\Delta_{\gB})$ be a compact quantum group acting continuously on $\qs(M_2)$. Let $\Psi_{\gB}\in\Mor(M_2,M_2\tens\gB)$ be the morphism describing this action. Then there exists an invariant state for $\Psi_{\gB}$. Indeed let $h$ be that Haar measure of $\cG$ and $\phi$ be any state on $M_2$. Then $\eta=(\phi\tens{h})\comp\Psi_{\gB}$ is an invariant state: for any $\in{M_2}$ and $\mu\in\gB'$
\[
\begin{split}
\mu\bigl((\eta\tens\id_{\gB})\Psi_{\gB}(m)\bigr)
&=\mu\bigl((\phi\tens{h}\tens\id_{\gB})(\Psi_{\gB}\tens\id_{\gB})\Psi_{\gB}(m)\bigr)\\
&=(\phi\tens{h}\tens\mu)(\id_{M_2}\tens\Delta_{\gB})\Psi_{\gB}(m)\\
&=(\phi\tens\id_{\gB})\bigl(\id\tens[h*\mu]\bigr)\Psi_{\gB}(m)\\
&=(\phi\tens\id_{\gB})(\id_{M_2}\tens{h})\Psi_{\gB}(m)\mu(\I)
=\mu\bigl(\eta(m)\I\bigr)
\end{split}
\]
(cf.~\cite[Lemma 4]{boca}).

States on $M_2$ is correspond to density matrices, so for $\eta$ there exists a unique $\rho$  such that $\eta(m)=\tr(\rho{m})$ for all $m$ and which is conjugate to one of the matrices
\begin{equation}\label{rhofam}
\rho_{q}=\tfrac{1}{1+q^2}\begin{bmatrix}1&0\\0&q^2\end{bmatrix}\qquad(q\in[0,1]).
\end{equation}
Note also that $\tr(\rho_q{m})=\om(m)$. Let $u$ be the unitary element of $M_2$ such that $\rho=u\rho_qu^*$. Then $\eta(umu^*)=\om(m)$ for all $m$. Therefore if we define
\[
\widetilde{\Psi}_{\gB}:M_2\ni{m}\longmapsto(u^*\tens\I)\Psi_{\gB}(umu^*)(u^*\tens\I)\in{M_2}\tens\gB
\]
we obtain an action $\widetilde{\Psi}_{\gB}\in\Mor(M_2,M_2\tens\gB)$: for all $m\in{M_2}$
\[
\begin{split}
\bigl(\widetilde{\Psi}_{\gB}\tens\id_{\gB}\bigr)\widetilde{\Psi}_{\gB}(m)
&=(u^*\tens\I\tens\I)\bigl[(\Psi_{\gB}\tens\id_{\gB})
\bigl((u\tens\I)\bigl(\widetilde{\Psi}_{\gB}(m)\bigr)(u^*\tens\I)\bigr)\bigr](u\tens\I\tens\I)\\
&=(u^*\tens\I\tens\I)\bigl[(\Psi_{\gB}\tens\id_{\gB})\Psi_{\gB}(umu^*)\bigr](u\tens\I\tens\I)\\
&=(u^*\tens\I\tens\I)\bigl[(\id_{M_2}\tens\Delta_{\gB})\Psi_{\gB}(umu^*)\bigr](u\tens\I\tens\I)\\
&=(\id_{M_2}\tens\Delta_{\gB})\bigl((u^*\tens\I)\Psi_{\gB}(umu^*)(u\tens\I)\bigr)=
(\id_{M_2}\tens\Delta_{\gB})\widetilde{\Psi}_{\gB}(m)
\end{split}
\]
which is continuous because
\[
\begin{split}
\bigl\{&\widetilde{\Psi}_{\gB}(m)(\I\tens{b})\st{m\in{M_2}},\;b\in\gB\bigr\}\\
&=\bigl\{(u^*\tens\I)\Psi_{\gB}(umu^*)(u\tens\I)(\I\tens{b}\st{m\in{M_2}},\;b\in\gB\bigr\}\\
&=(u^*\tens\I)\bigl\{\Psi_{\gB}(k)(\I\tens{b}\st{k\in{M_2}},\;b\in\gB\bigr\}(u\tens\I)
\end{split}
\]
is linearly dense in $M_2\tens\gB$. Moreover for any $m\in{M_2}$ we have
\[
\begin{split}
(\om\tens\id_{\gB})\widetilde{\Psi}_{\gB}(m)
&=(\om\tens\id_{\gB})\bigl((u^*\tens\I)\Psi_{\gB}(umu^*)(u\tens\I)\bigr)\\
&=(\eta\tens\id_{\gB})\widetilde{\Psi}_{\gB}(umu^*)\\
&=\eta(umu^*)\I=\om(m)\I,
\end{split}
\]
so $\widetilde{\Psi}$ preserves $\om$.

In order to state the next theorem in a readable way let us use the symbol $\mathrm{S}_0\mathrm{O}(3)$ for the group $\TT$. For each $q\in[0,1]$ the symbol $\Psi_q$ denotes the action of $\SqO$ on $\qs(M_2)$ described in Subsections \ref{Act}, \ref{q1} and \ref{q0}.

Applying the results of Sections \ref{char} and \ref{q10} we get the following description of all continuous actions of compact quantum groups on $M_2$. 

\begin{thm}\label{dzialania}
Let $\cG=(\gB,\Delta_{\gB})$ be a compact quantum group and let $\Psi_{\gB}\in\Mor(M_2,M_2\tens\gB)$ be a continuous action of $\cG$ on $\qs(M_2)$. Then there exists a unitary $u\in{M_2}$, $q\in[0,1]$ and $\Gamma\in\Mor\bigl(\C(\SqO),\gB)$ such that
\[
\Psi_{\gB}(m)=(\id_{M_2}\tens\Gamma)\bigl((u\tens\I)\Psi_q(u^*mu)(u^*\tens\I)\bigr).
\]
$\Gamma$ is unique for each fixed pair $(q,u)$. Moreover if $\Psi_{\gB}$ is ergodic then $q$ and $u$ are unique.
\end{thm}

The only element of Theorem \ref{dzialania} which requires a comment at this stage is the uniqueness statement. It follows from the fact that for an ergodic continuous action of a compact quantum group there exists a unique invariant state (\cite[Lemma 4]{boca}) whose density matrix is then conjugate to a unique matrix from the family \eqref{rhofam}.


\begin{thebibliography}{6}
\bibitem{boca}
{\sc F.~Boca:} Ergodic actions of compact matrix pseudogroups on $\mathrm{C}^*$-algebras. In \emph{Recent Advances in Operator algebras.} Ast\'{e}risque \textbf{232} (1995), pp.~93--109.
\bibitem{izumi}
{\sc M.~Izumi:} Non commutative Poisson boundaries and compact quantum group actions. \emph{Adv.~Math.} \textbf{169} (2002), 1--57.
\bibitem{marciniak}
{\sc M.~Marciniak:} Quantum symmetries in noncommutative $\mathrm{C}^*$-systems. \emph{Quantum probability} Banach Center Publications 1998, pp.~297--307.
\bibitem{spheres}
{\sc P.~Podle\'s:} Quantum spheres. \emph{Lett.~Math.~Phys.} \textbf{14} (1987), 193--202.
\bibitem{podPHD}
{\sc P.~Podle\'s:} Przestrzenie kwantowe i ich grupy symetrii (Quantum spaces and their symmetry groups). PhD Thesis, Department of Mathematical Methods in Physics, Faculty of Physics, Warsaw University (1989) (in Polish).
\bibitem{podles}
{\sc P.~Podle\'{s}:} Symmetries of quantum spaces. Subgroups and quotient spaces of quantum $\mathrm{SU}(2)$ and $\mathrm{SO}(3)$ groups. \emph{Commun.~Math.~Phys.} \textbf{170} (1995), 1--20.
\bibitem{PodMu}
{\sc P.~Podle\'s \& E.~M\"uller:} Introduction to quantum groups. \emph{Rev.~Math.~Phys.} \textbf{10} no.~4 (1998), 511--551.
\bibitem{rud}
{\sc W.~Rudin:} \emph{Functional analysis.} McGraw-Hill 1973.
\bibitem{qs}
{\sc P.M.~So{\l}tan:} Quantum families of maps and quantum semigroups on finite quantum spaces.
Archive: arXiv:math/0610922v5 [math.OA], submitted to \emph{J.~Geom.~Phys.}
\bibitem{kom}
{\sc P.M.~So{\l}tan:} Examples of quantum commutants. Accepted to \emph{Arabian Journal of Science and Engeneering, theme issue on Applications of Algebraic \& Coalgebraic Structures.} Archive:    arXiv:math/0806.0503v2 [math.QA].
\bibitem{vaes}
{\sc S.~Vaes:} A new approach to induction and imprimitivity results. \emph{J.~Funct.~Anal.} \textbf{229} (2005), 317--374.
\bibitem{wang}
{\sc S.~Wang:} Quantum symmetry groups of finite spaces. \emph{Commun.~Math.~Phys.} \textbf{195} (1998), 195--211.
\bibitem{pseu}
{\sc S.L.~Woronowicz:} Pseudogroups, pseudospaces and Pontryagin duality.
\emph{Proceedings of the International Conference on Mathematical Physics,
Lausanne 1979} Lecture Notes in Physics, \textbf{116}, pp.~407--412.
\bibitem{su2}
{\sc S.L.~Woronowicz:} Twisted $\mathrm{SU}(2)$ group. An example of noncommutative differential calculus. \emph{Publ.~RIMS, Kyoto University} \textbf{23} (1987), 117--181.
\bibitem{pseudogr}
{\sc S.L.~Woronowicz:} Compact matrix pseudogroups. \emph{Comm.~Math.~Phys.} \textbf{111} (1987), 613--665.
{\bibitem{unbo}
{\sc S.L.~Woronowicz:} Unbounded elements affiliated with $\mathrm{C}^*$-algebras and non-compact quantum groups. \emph{Commun.~Math.~Phys.} \textbf{136} (1991), 399--432.}
\bibitem{cqg}
{\sc S.L.~Woronowicz:} Compact quantum groups. In: \emph{Sym\'etries
quantiques, les Houches, Session LXIV 1995,} Elsevier 1998, pp.~845--884, Elsevier 1998.
\bibitem{qexp}
{\sc S.L.~Woronowicz:} Quantum exponential function. \emph{Rev.~Math.~Phys.} \textbf{12} no.~6 (2000), 873--920
\end{thebibliography}
\end{document}